\documentclass[12pt,letterpaper]{article}
\usepackage{latexsym}
\usepackage{amssymb}
\usepackage{amsmath}
\usepackage{amsthm}

\usepackage[all]{xy}
\usepackage{sectsty} 
\newtheorem{theorem}{Theorem}[section]

\newtheorem{lemma}[theorem]{Lemma}

\newtheorem{conjecture}[theorem]{Conjecture}
\newtheorem{question}[theorem]{Question}
\newtheorem{remark}[theorem]{Remark}

\theoremstyle{definition}
\newtheorem{definition}[theorem]{Definition}

\newcommand{\map}{$f: X \longrightarrow Y$ }

\newcommand{\ox}{${\cal O}_{X,p}$ }
\newcommand{\oox}{${\cal O}_X$ }

\newcommand{\oxc}{${\cal \hat{O}}_{X,p}$ }
\newcommand{\oy}{${\cal O}_{Y,q}$ }

\newcommand{\x}{$x_1,...,x_n$}
\newcommand{\xx}{$\bar{x_1},...,\bar{x_n}$}
\sectionfont{\fontfamily{ptm}\selectfont}
\allsectionsfont{\mdseries \centering \scshape}
\makeatletter
\def\@seccntformat#1{\csname the#1\endcsname.\quad}
\makeatother

\begin{document}
\title{Toroidalization of Locally Toroidal Morphisms from N-folds to Surfaces}
\author{Krishna Hanumanthu}
\date{}
\maketitle

\section{Introduction}

Fix an algebraically closed field $k$ of characteristic 0. A variety is an open 
subset of an irreducible proper $k$-scheme.

A {\it simple normal crossing} (SNC) divisor on a nonsingular variety
is a divisor $D$ on $X$, all of whose irreducible 
components are nonsingular and
whenever $r$ irreducible components $Z_1,...,Z_r$ of $D$ meet at a point $p$, 
then local equations $x_1,...,x_r$ of $Z_i$ form part of a regular system
of parameters in ${\cal O}_{X,p}$.

If $D$ is a SNC divisor and a point
$p \in D$ belongs to exactly $k$ components of $D$, then we say that $p$ is a 
$k$ point.

A {\it toroidal structure} on a nonsingular variety $X$ is a 
SNC divisor $D_X$. 

The divisor $D_X$ specifies a \textit{toric} chart $(V_p, \sigma_p)$ 
at every closed point $p \in X$ where  $p \in V_p \subset X$ is
an open neighborhood and $\sigma_p: V_p \rightarrow X_p$
is an \'{e}tale morphism to a toric variety $X_p$ such that under
$\sigma_p$
the ideal of $D_X$ at $p$ corresponds to the ideal of the 
complement of the torus in $X_p$.


The idea of a toroidal structure is 
fundamental to algebraic geometry. It is developed in the classic book
``Toroidal Embeddings I" \cite{KKMS}
by G. Kempf, F. Knudsen, D. Mumford and 
B. Saint-Donat.

\begin{definition} \label{defn} (\textup{\cite{KKMS}, \cite{AK}})
Suppose that $D_X$ and $D_Y$ are
toroidal structures on $X$ and $Y$ respectively. 
Let $p \in X$ be a closed point.
A dominant
morphism $f:X \rightarrow Y$ is 
\textit{toroidal at $p$ (with respect to the toroidal structures 
$D_X$ and $D_Y$)} 
if the germ of $f$ at $p$ is formally isomorphic to a 
toric morphism between the toric charts 
at $p$ and $f(p)$. 
$f$ is
\textit{toroidal} 
if it is toroidal at all closed points
in $X$. 
\end{definition}


A nonsingular subvariety $V$ of $X$ 
is a {\it possible center} for $D_X$ if $V \subset D_X$ and $V$ 
intersects $D_X$ transversally. That is, 
$V$ makes $SNCs$ with $D_X$, as defined before Lemma \ref{ni}. 
The blowup $\pi : X_1 \rightarrow X$ of a possible center 
is called a possible blowup. $D_{X_1} = {\pi}^{-1}(D_X)$ is then a 
toroidal structure on $X_1$. 

Let $Sing(f)$ be the set of points $p$ in $X$ where $f$ is not smooth. It 
is a closed set.

The following  
``toroidalization conjecture" is the strongest possible general
structure theorem for morphisms of varieties.

\begin{conjecture}\label{main}
Suppose that \map is a dominant morphism of nonsingular varieties. 
Suppose also that there is a SNC divisor $D_Y$ on $Y$ such that 
$D_X = f^{-1}(D_Y)$ is a SNC divisor on $X$ which contains the 
singular locus, $Sing(f)$, of the map $f$. 

Then there exists a commutative diagram of morphisms 
\begin{displaymath}
\xymatrix @R=3pc @C=3pc{
X_1 \ar[r]^{f_1} \ar[d]^{\pi_1} & Y_1 \ar[d]^{\pi} \\
X  \ar[r]^f    & Y }
\end{displaymath}
where $\pi$, $\pi_1$ are possible blowups for the preimages of $D_Y$ and 
$D_X$ respectively, such that $f_1$ is toroidal with respect to 
$D_{Y_1} = {\pi}^{-1}(D_Y)$ and $D_{X_1} = {\pi_1}^{-1}(D_X)$
\end{conjecture}

A slightly weaker version of the conjecture is stated in the paper 
\cite{AKMW} of D. Abramovich, K. Karu, K. Matsuki, and J. Wlodarczyk.

When $Y$ is a curve, this conjecture follows easily from embedded resolution
of hypersurface singularities, as shown in the introduction of \cite{C1}. 
The case when $X$ and $Y$ are surfaces has been
known before (see Corollary 6.2.3 \cite{AKMW}, 
\cite{AkK}, \cite{CP}). The case 
when $X$ has dimension 3 is completely resolved 
by Dale Cutkosky
in \cite{C1} and \cite{C6}. A special case of dim$(X)$ arbitrary 
and dim$(Y) = 2$ is done in \cite{CK}.

For detailed history and applications of this conjecture,
see \cite{C6}.


A related, but weaker question asked by Dale Cutkosky is 
the following Question \ref{our}. 





To state the question we need the following definition.

\begin{definition}\label{locally toroidal} 
Let $f: X \rightarrow Y$ be a dominant morphism of nonsingular varieties. 
Suppose that the following are true.
\begin{enumerate}
\item There exist open coverings
 $\{U_1,...,U_m\}$ and $\{V_1,...,V_m\}$ 
of $X$ and $Y$ respectively
such that the morphism $f$ restricted to $U_i$
maps into $V_i$ for all $i = 1,...,m$. 
\item There exist simple normal crossings divisors
$D_i$ and $E_i$  
in $U_i$ and $V_i$ respectively such that 
$f^{-1}(E_i) \cap U_i = D_i$ and $Sing({f|}_{U_i}) \subset D_i$ for all 
$i = 1,...,m$.
\item 
The restriction of $f$ to $U_i$,
${f|}_{U_i}: U_i \rightarrow V_i$, is toroidal with respect to $D_i$ and $E_i$
for all $i = 1,...,m$.
\end{enumerate}
Then we say that $f$ is {\it locally toroidal} with respect to 
the open coverings $U_i$ and $V_i$ and SNC divisors $D_i$ and $E_i$.
\end{definition}

For the remainder when we say ``$f$ is locally toroidal", it is to be understood that 
$f$ is locally toroidal with respect to 
the open coverings $U_i$ and $V_i$ and SNC divisors $D_i$ and $E_i$ as in the 
definition.
We will usually not mention $U_i$, $V_i$, $D_i$ and $E_i$.


We have the following.
\begin{question}\label{our}
Suppose that \map is locally toroidal. Does there exist 
a commutative diagram of morphisms 
\begin{displaymath}
\xymatrix @R=3pc @C=3pc{
X_1 \ar[r]^{f_1} \ar[d]^{\pi_1} & Y_1 \ar[d]^{\pi} \\
X  \ar[r]^f    & Y }
\end{displaymath}
where $\pi$, $\pi_1$ are blowups of nonsingular varieties such that 
there exist 
SNC 
divisors $E$, $D$ on $Y_1$ and $X_1$ respectively such that 
$Sing(f_1) \subset D$, ${f_1}^{-1}(E) = D$ and $f_1$
is toroidal with respect to $E$ and $D$?  
\end{question}

The aim of this paper is to give a positive answer
to this question when $Y$ 
is a surface and $X$ is arbitrary. 
The result is proved in Theorem \ref{Final}.\\

\noindent
{\bf Brief outline of the proof:}

The core results (Theorems \ref{final} and \ref{Final}) are proved in 
section 4. Sections 2 and 3 consist of preparatory material.   

Let $f: X \rightarrow Y$ be a locally toroidal morphism with the 
notation as in definition \ref{locally toroidal}. The essential
observation is this: if there is a SNC divisor $E$ on $Y$ such that 
$E_i \subset E$ for all $i$, then $f$ is toroidal with respect to 
$E$ and $f^{-1}(E)$. A proof of this observation is contained in the proof
of Theorem \ref{Final}. 

The main task, then, is to construct the divisor $E$. This is not hard: 
consider the divisor $E' = \bar{E_1} + ... + \bar{E_m}$ where
$\bar{E_i}$ is the Zariski closure of $E_i$ in
$Y$. By embedded resolution of 
singularities, there exists a finite sequence of
blowups of points $\pi: Y_1 \rightarrow Y$ such that 
${\pi}^{-1}(E^{\prime})$ 
is a 
SNC divisor on $Y_1$. 

The problem now reduces to constructing a sequence of blowups 
$\pi_1: X_1 \rightarrow X$ such that there is a locally toroidal morphism 
$f_1: X_1 \rightarrow Y_1$. This is done in Theorem \ref{final}.

Sections 2 and 3 prepare the ground for Theorem \ref{final}.

Given the sequence of blowups $\pi: Y_1 \rightarrow Y$ as above, there 
exist principalization algorithms which give a sequence of blowups 
$\pi_1: X_1 \rightarrow X$ so that there exists a morphism 
$f_1: X_1 \rightarrow Y_1$. The main difficulty we face is that such
a morphism $f_1$ may not be locally toroidal. So a blanket appeal to existing
principalizing algorithms can not be made. In sections 2 and 3, we 
construct a specific algorithm that works in our situation.

Section 2 deals with the blowups that preserve the local toroidal structure.
We call these {\it{permissible blowups}} (definition \ref{permissible}). The
main result of section 2 is Lemma \ref{a} which analyzes the effect of a
permissible sequence of blowups.

In section 3, we define invariants on nonprincipal locus of the morphism
$f$. These invariants are positive integers and we prescribe permissible
sequences of blowups under which these invariants drop 
(Theorems \ref{1 point} and \ref{2 point}). Finally we achieve 
principalization in Theorem \ref{principalization}.

\section{Permissible Blowups}
Let \map be a locally toroidal morphism from a nonsingular
$n$-fold $X$ to a nonsingular surface $Y$ with respect to 
open coverings $\{U_1,...,U_m\}$ and $\{V_1,...,V_m\}$ 
of $X$ and $Y$ respectively and SNC divisors 
$D_i$ and $E_i$  
in $U_i$ and $V_i$ respectively. Then we have the following

\begin{lemma}
Let $p \in D_i$. Then there exist regular parameters 
$x_1,...,x_n$ in ${\cal \hat{O}}_{X,p}$ and regular
parameters $u,v$ in ${\cal O}_{Y,q}$ such
that one of the following forms holds:

$1 \leq k \leq n-1:$ $u = 0$ is a local equation
of $E_i$, $x_1...x_k = 0$ is a local equation of $D_i$
and 
\begin{eqnarray}\label{t1}
u = {x_1}^{a_1}...{x_k}^{a_k}, ~~~ v = x_{k+1},
\end{eqnarray} 
where
$a_1,...,a_k > 0$. 

$1  \leq k \leq n-1:$ $uv = 0$ is a local
equation for $E_i$, $x_1...x_k = 0$ is a local
equation of $D_i$ and
\begin{eqnarray}\label{t2}
u = {({x_1}^{a_1}...{x_k}^{a_k})}^m, ~ v =
{({x_1}^{a_1}...{x_k}^{a_k})}^t(\alpha+x_{k+1}),
\end{eqnarray}
where 
$a_1,...,a_k, m, t> 0$ and $\alpha \in K - \{0\}$.

$2 \leq k \leq n:$ $uv = 0$ is a local equation
of $E_i$, $x_1...x_k = 0$ is a local equation of $D_i$
and 
\begin{eqnarray}\label{t3}
u = {x_1}^{a_1}...{x_k}^{a_k}, ~~ v =
{x_1}^{b_1}...{x_k}^{b_k},
\end{eqnarray}
where  $a_1,...,a_k,b_1,...,b_k
\geq 0, a_i+b_i > 0$ for all $i$ and \\
rank $\left[
\begin{array}{cccc} a_1 & .&.& a_k \\b_1 &.&.& b_k \end{array}
\right] = 2 $.
\end{lemma}
\begin{proof}
This follows from Lemma 4.2 in \cite{CK}.
\end{proof}

\begin{definition}
Suppose that $D$ is a SNC divisor on a variety $X$, and $V$ is a nonsingular subvariety 
of $X$. We say that $V$ makes SNCs with $D$ at $p \in X$ if there exist regular parameters
$x_1,...,x_n$ in \ox and $e, r \leq n$ such that $x_1...x_e=0$ is a local equation of $D$ at $p$
and $x_{\sigma(1)}=...=x_{\sigma(r)}=0$ is a local equation of $V$ at $p$ for some injection
$\sigma: \{1,...,r\} \rightarrow \{1,...,n\}$.

We say that $V$ makes SNCs with $D$ if $V$ makes SNCs with $D$ at all points $p \in X$.
\end{definition}

Let $q \in Y$ and let $m_q$ be the maximal ideal
of \oy. 

Define $W_q = \{p \in X~ | ~m_q$\ox is not
principal$\}$. Note that the closed subset 
$W_q \subset f^{-1}(q)$ and that $m_q$\ox is 
principal if and only if $m_q$\oxc is principal.

\begin{lemma}\label{ni} For all $q \in Y$, $W_q$ is a union of nonsingular codimension 2 subvarieties
of $X$, which make SNCs with each divisor $D_i$ on $U_i$. 
\end{lemma}

\begin{proof}
Let us fix a $q \in Y$ and denote $W = W_q$. Let $\mathfrak{I}_W$ be the
reduced ideal sheaf of $W$ in $X$, and let $\mathfrak{I}_q$ 
be the reduced ideal sheaf of $q$ in $Y$.

Since the conditions that $W$ is nonsingular and has codimension 2 in $X$ are
both local properties, we need only check that for all $p \in W$,
$\mathfrak{I}_{W,p}$ is an intersection of height
2 prime ideals which are regular.

Since $X$ is nonsingular, $\mathfrak{I}_{q}$\oox = \oox$(-F)\cal{I}$ where $F$
is an effective Cartier divisor on $X$ and $\cal{I}$ is an ideal sheaf such that
the support of ${\cal{O}}_X/\cal{I}$ has codimension at least 2 on $X$. We have
$W = $ supp(${\cal{O}}_X/\cal{I}$). The ideal sheaf of $W$ is 
$\mathfrak{I}_W = \sqrt{\cal{I}}$.

Let $p \in W$. We have that $p \in U_i$ for some $1 \leq i \leq m$. 

Suppose first that $q \notin E_i$. Then $f$ is smooth at $p$ because
it is locally toroidal. 
This means that there are regular parameters $u,v$ at $q$ 
which form a part of a regular sequence at $p$. So we have 
regular parameters \x~   in ${\cal {O}}_{X,p}$ such that 
$u = x_1, v = x_2.$  

$\mathfrak{I}_{q}$\ox = $(u,v)$\ox = $(x_1,x_2)$\ox. It follows that 
$\mathfrak{I}_{W,p} = (x_1,x_2)$${\cal O}_{X,p}$. This gives us the lemma. 

Suppose now that $q \in E_i$. 

Since $p \in W_q$, 
there exist regular
parameters $x_1,...,x_n$ in \oxc and $u,v$ in \oy such that one of the forms
(\ref{t1}) or (\ref{t3}) holds.

Suppose that (\ref{t1}) holds. Since $D_j$ is a SNC divisor, there exist regular
parameters $y_1,...,y_n$ in \ox and some $e$ such that $y_1...y_e = 0$ is a
local equation of $D_j$.

Since $x_1...x_k = 0$ is a local equation
for $D_j$ in \oxc, there exists a unit series $\delta \in$ \oxc such that
$y_1...y_e = \delta x_1...x_k$. 
Since the $x_i$ and $y_i$ are irreducible in \oxc, 
it follows that $e = k$, and there exist
unit series $\delta_i \in$ \oxc such that $x_i = \delta_iy_i$ for
$1 \leq i \leq k$, after possibly reindexing the $y_i$.

Note that $y_1,...,y_k,x_{k+1},y_{k+2},...,y_n$ is a regular system of
parameters in \oxc, after possibly permuting $y_{k+1},...,y_n$. 

So the ideal
$(y_1,...,y_k,x_{k+1},y_{k+2},...,y_n)$\oxc is the maximal ideal of 
${\cal \hat{O}}_{X,p}$.
Since $x_{k+1} = v \in$ \ox, $y_1,...,y_k,x_{k+1},y_{k+2},...,y_n$
generate an ideal $J$ in ${\cal O}_{X,p}$.  
Since \oxc is faithfully flat over ${\cal O}_{X,p}$, and 
$J{\cal \hat{O}}_{X,p}$ is maximal,
it follows that $J$ is the maximal ideal of ${\cal O}_{X,p}$.
Hence $y_1,...,y_k,x_{k+1},y_{k+2},...,y_n$ is a regular system of parameters in 
${\cal O}_{X,p}$.

Rewriting (\ref{t1}), we have $u = {y_1}^{a_1}...{y_k}^{a_k} \bar{\delta}$, 
where $\bar{\delta}$ is a unit in ${\cal \hat{O}}_{X,p}$. 

Since
$\bar{\delta} = \frac{u}{{y_1}^{a_1}...{y_k}^{a_k}}$, 
$\bar{\delta} \in$ QF$({\cal O}_{X,p}) \cap$ \oxc, where
QF$({\cal O}_{X,p})$  is  the quotient field of ${\cal O}_{X,p}$.
By Lemma 2.1 in \cite{C}, it follows that $\bar{\delta} \in$ \ox.

Since $\bar{\delta}$ is a unit in ${\cal \hat{O}}_{X,p}$, 
it is a unit in ${\cal O}_{X,p}$.

We have 

\begin{eqnarray*}
\mathfrak{I}_{W,p} &=& \sqrt{\mathfrak{I}_q {\cal O}_{X,p}} 
= \sqrt{(u,v){\cal O}_{X,p}}  
= \sqrt{({y_1}^{a_1}...{y_k}^{a_k},x_{k+1})} \\
&=& 
(y_1,x_{k+1})
\cap (y_2,x_{k+1}) 
\cap ... \cap (y_{k},x_{k+1}),
\end{eqnarray*}
as required.





We argue similarly when (\ref{t3}) holds at $p$.
\end{proof}

Let $Z$ be a nonsingular codimension 2 subvariety of $X$ such that
$Z \subset W_q$ for some $q$. Let $\pi_1: X_1 \rightarrow X$ be the blowup of $Z$.
Denote by $(W_1)_q$ the set $\{p \in X_1~ |~ m_q{\cal \hat{O}}_{X_1,p}$
is not invertible$\}$.

Given any sequence of blowups $X_n \rightarrow X_{n-1} \rightarrow
... \rightarrow X_1 \rightarrow X$, we 
define $(W_i)_q$ for each $X_i$ as above.

\begin{definition} \label{permissible}
Let $q \in Y$. A sequence of blowups $X_k \rightarrow X_{k-1}
\rightarrow ... \rightarrow X_1 \rightarrow X$ is called a
{\it permissible sequence with respect to $q$} if 
for all $i$, 
each blowup $X_{i+1} \rightarrow X_i$ is
centered at a nonsingular codimension 2 subvariety $Z$ of $X_i$ such
that $Z \subset (W_i)_q$.
\end{definition}

We will often write simply permissible sequence without mentioning
$q$
if there is no scope for confusion. 

\begin{lemma}\label{a} Let $f: X \rightarrow Y$ be a locally toroidal
morphism.
Let $\pi_1: X_1 \rightarrow X$ be a permissible sequence with respect to a
$q \in Y$. \\

\noindent
{\bf I} Suppose that $1 \leq i \leq m$ 
and $p \in {(f \circ \pi_1)}^{-1}(q) \cap {\pi_1}^{-1}(U_i)$
and $q \in E_i$. Then 
{\bf I.A} and {\bf I.B} as below hold.\\

\noindent
{\rm \bf I.A}. There exist regular parameters \x~   in ${\cal \hat{O}}_{X_1,p}$ and
$(u,v)$ in \oy such that one of the following forms holds:

\begin{itemize}
\item[] 1 $\leq k \leq n-1$: $u = 0$ is a local
equation of $E_i$, $x_1...x_k = 0$ is a local equation of
${\pi_1}^{-1}(D_i)$ and 
\begin{eqnarray}\label{1}
u = {x_1}^{a_1}...{x_k}^{a_k},  v =
{x_1}^{b_1}...{x_k}^{b_k}x_{k+1},
\end{eqnarray}
where $b_i \leq a_i$.

\item[]1 $\leq k \leq n-1$: $u = 0$ is a local
equation of $E_i$, $x_1...x_kx_{k+1} = 0$ is a local equation of
${\pi_1}^{-1}(D_i)$ and 
\begin{eqnarray}\label{2}
u = {x_1}^{a_1}...{x_k}^{a_k}{x_{k+1}}^{a_{k+1}},  v =
{x_1}^{b_1}...{x_k}^{b_k}{x_{k+1}}^{b_{k+1}},
\end{eqnarray}
where $b_i \leq a_i$
for $i = 1,...,k$ and $b_{k+1} < a_{k+1}$.

\item[] 1 $\leq k \leq n-1$: $u = 0$ is a local
equation of $E_i$, $x_1...x_k = 0$ is a local equation of
${\pi_1}^{-1}(D_i)$ and 
\begin{eqnarray}\label{3}
u = {x_1}^{a_1}...{x_k}^{a_k},  v =
{x_1}^{b_1}...{x_k}^{b_k}(x_{k+1}+\alpha),
\end{eqnarray}
where $b_i \leq a_i$
for all $i$ and $0 \neq \alpha \in K$.

\item[] 1 $ \leq k \leq n-1$: $uv = 0$ is a local
equation for $E_i$, $x_1...x_k = 0$ is a local equation of
${\pi_1}^{-1}(D_i)$ and 
\begin{eqnarray}\label{4}
u = {({x_1}^{a_1}...{x_k}^{a_k})}^m,  v =
{({x_1}^{a_1}...{x_k}^{a_k})}^t(\alpha+x_{k+1}), 
\end{eqnarray}
where
$a_1,...,a_k, m, t> 0$ and $\alpha \in K - \{0\}$.
\item[]2 $\leq k \leq n$: $uv = 0$ is a local
equation of $E_i$, $x_1...x_k = 0$ is a local equation of
${\pi_1}^{-1}(D_i)$ and 
\begin{eqnarray}\label{5}
u = {x_1}^{a_1}...{x_k}^{a_k}, v =
{x_1}^{b_1}...{x_k}^{b_k},
\end{eqnarray} 
where $a_1,...,a_k,b_1,...,b_k \geq 0$,
$a_i+b_i > 0$ for all $i$ and rank $\left[ \begin{array}{cccc} a_1 &
.&.& a_k \\b_1 &.&.& b_k \end{array} \right]$= 2.
\end{itemize}

\noindent
{\rm \bf I.B}. Suppose that $p_1 \in (W_1)_q$. 
There exist regular parameters \x~   in ${\cal \hat{O}}_{X_1,p}$ and
$(u,v)$ in \oy such that one of the following forms holds:
\begin{itemize}
\item[] 1 $\leq k \leq n-1$: $u = 0$ is a local
equation of $E_i$, $x_1...x_k = 0$ is a local equation of
${\pi_1}^{-1}(D_i)$ and 
\begin{eqnarray}\label{n1}
u = {x_1}^{a_1}...{x_k}^{a_k},  v =
{x_1}^{b_1}...{x_k}^{b_k}x_{k+1}, 
\end{eqnarray}
where $b_i \leq a_i$ and $b_i < a_i$ for
some $i$.
Moreover, the local equations of $(W_1)_q$ are 
$x_i = x_{k+1}= 0$ where $b_i < a_i$.

\item[]2 $\leq k \leq n$: $uv = 0$ is a local
equation of $E_i$, $x_1...x_k = 0$ is a local equation of
${\pi_1}^{-1}(D_i)$ and 
\begin{eqnarray}\label{n2}
u = {x_1}^{a_1}...{x_k}^{a_k}, v =
{x_1}^{b_1}...{x_k}^{b_k},
\end{eqnarray}
where $a_1,...,a_k,b_1,...,b_k \geq 0$,
$a_i+b_i > 0$ for all $i$, $u$ does not divide $v$, $v$ does not divide $u$, and rank $\left[ \begin{array}{cccc} a_1 &
.&.& a_k \\b_1 &.&.& b_k \end{array} \right]$= 2.
Moreover, the local equations of $(W_1)_q$ are 
$x_i = x_j= 0$ where $(a_i-b_i)(b_j-a_j) > 0$.
\end{itemize}
{\bf II} Suppose that $1 \leq i \leq m$ and $p \in {(f \circ \pi_1)}^{-1}(q) \cap {\pi_1}^{-1}(U_i)$
and $q \notin E_i$. Then 
{\bf II.A} and {\bf II.B} as below hold.\\

\noindent
{\bf II.A} There exist regular parameters \x~   in ${\cal \hat{O}}_{X_1,p}$ and
$(u,v)$ in \oy such that one of the following forms holds:

\begin{eqnarray}\label{6}
u = x_1, v = x_2
\end{eqnarray}

\begin{eqnarray}\label{7}
u = x_1, v = x_1(x_2+\alpha) ~\text {for some} ~\alpha \in K. 
\end{eqnarray}

\begin{eqnarray}\label{8} 
u = x_1x_2, v = x_2.
\end{eqnarray}
\noindent
{\bf II.B} Suppose that $p_1 \in (W_1)_q$. 
There exist regular parameters \x~   in ${\cal \hat{O}}_{X_1,p}$ and
$(u,v)$ in \oy such that the following form holds:

\begin{eqnarray}\label{n3}
u = x_1, v = x_2. 
\end{eqnarray}
The local equations of $(W_1)_q$ are 
$x_1 = x_2 = 0$.\\

\noindent
{\bf III} $(W_1)_q$ is a union of nonsingular codimension 2 subvarieties of
$X_1$. 
\end{lemma}
\begin{proof}

\noindent

{\bf I}
We prove this part by induction on the number of blowups in the 
sequence $\pi_1: X_1 \rightarrow X$. In
$X$ the conclusions hold because of Lemma \ref{ni} and $f$ is locally 
toroidal.
Suppose that the conclusions of the lemma
hold after any sequence of $l$ permissible blowups where $l \geq 0$.  

Let $\pi_1: X_1 \rightarrow X$ be a permissible sequence (with respect to
$q$) of $l$ blowups. Let $\pi_2: X_2 \rightarrow X_1$ be the blowup of 
a nonsingular codimension 2 subvariety $Z$ of $X_1$ such that $Z \subset (W_1)_q$. 

Let $p \in {\pi_2}^{-1}({\pi_1}^{-1}(U_i)) \cap {(f \circ \pi_1 \circ \pi_2)}^{-1}(q)$ for 
some $1 \leq i \leq m$.

If $p_1 = \pi_2(p) \notin Z$ then $\pi_2$ is an isomorphism at $p$ and
we have nothing to
prove. Suppose then that $p_1 \in {\pi_1}^{-1}(U_i) \cap Z \subset 
{\pi_1}^{-1}(U_i) \cap (W_1)_q$.

Then by induction hypothesis ({\bf I.B}) 
$p_1$ has the form (\ref{n1}) or (\ref{n2}). Suppose
first that it has the form (\ref{n1}). 

Then the local equations of $Z$
at $p_1$ are $x_i = x_{k+1} = 0$ for some $1 \leq i \leq k$. Note
that $b_i < a_i$. 

As in the proof of Lemma \ref{ni}, there exist regular parameters 
$y_1,...,y_k,$ $x_{k+1},y_{k+2},...,y_n$ in ${\cal O}_{X_1,p_1}$ and
unit series $\delta_i \in $${\cal \hat{O}}_{X_1,p_1}$ such that 
$y_i = \delta_i x_i$ for $1 \leq i \leq k$.

Then ${\cal {O}}_{X_2,p}$ has one of the following two forms: 

\begin{itemize}
\item[(a)] ${\cal {O}}_{X_2,p} = {{\cal {O}}_{X_1,p_1}[\frac{x_{k+1}}{y_i}]}_{(y_i,\frac{x_{k+1}}{y_i}-\alpha)}$ 
for some $\alpha \in K$, or 
\item[(b)] ${\cal {O}}_{X_2,p} = {{\cal {O}}_{X_1,p_1}[\frac{y_i}{x_{k+1}}]}_{(x_{k+1},\frac{y_i}{x_{k+1}})}$
\end{itemize}

In case(a), set $\bar{y}_{k+1} = \frac{x_{k+1}}{y_i}-\alpha$. Then 
$y_1,...,y_k,\bar{y}_{k+1},y_{k+2},...,y_n$ are regular parameters in 
${\cal {O}}_{X_2,p}$ and so
${\cal \hat {O}}_{X_2,p} = k[[y_1,...,y_k,\bar{y}_{k+1},y_{k+2},...,y_n]].$

Let $c \neq 0$ be the constant term of the unit series $\delta_i$. 

Then evaluating $\delta_i$ in the local ring ${\cal {O}}_{X_2,p}$ we get,
\begin{eqnarray*}
  \delta_i(y_1,...,y_k,x_{k+1},y_{k+2},...,y_n) 
&=& \delta_i(y_1,...,y_k,y_i(\bar{y}_{k+1}+\alpha),y_{k+1},...,y_n) \\
& = & c + \Delta_1 y_1 + ... + \Delta_k y_k + \Delta_{k+2} y_{k+2} + ... + \Delta_n y_n
\end{eqnarray*}
for some $\Delta_i \in {\cal {O}}_{X_2,p}$. 

Set $\bar{\alpha} = c\alpha$. Note that 
$\frac{x_{k+1}}{x_i} - \bar{\alpha} = 
\delta_i \frac{x_{k+1}}{y_k} - c\alpha = \delta_i(\bar{y}_{k+1}+\alpha) - c\alpha
= \delta_i \bar{y}_{k+1} + (\delta_i - c)\alpha.$

Since $y_1,...,y_k,\bar{y}_{k+1},y_{k+2},...,y_n$ are regular parameters in 
${\cal \hat {O}}_{X_2,p}$
the above calculations imply that 
$x_1,...,x_k,\frac{x_{k+1}}{x_i} - \bar{\alpha},y_{k+2},...,y_n$
are regular parameters in ${\cal \hat {O}}_{X_2,p}$. 

Set $\bar{x}_{k+1} = \frac{x_{k+1}}{x_k}-\bar{\alpha}$.

We get $u = {x_1}^{a_1}...{x_k}^{a_k}, ~ v =
{x_1}^{b_1}...{\bar{x_i}}^{b_i+1}...{x_k}^{b_k}(\bar{x}_{k+1}+\alpha).$

This is the form (\ref{3}) if $\alpha \neq 0$ and form (\ref{1}) if
$\alpha = 0$.

In case (b), set $\bar{y}_{k+1} = \frac{y_i}{x_{k+1}}$. Then 
$y_1,...,y_k,\bar{y}_{k+1},y_{k+2},...,y_n$ are regular parameters in 
${\cal {O}}_{X_2,p}$ and so
${\cal \hat {O}}_{X_2,p} = k[[y_1,...,y_k,\bar{y}_{k+1},y_{k+2},...,y_n]].$

Then $x_1,...,x_k,\frac{x_i}{x_{k+1}},y_{k+2},...,y_n$ are regular parameters 
in ${\cal \hat {O}}_{X_2,p}$. Set $\bar{x}_i = \frac{x_i}{x_{k+1}}$.

$u = {x_1}^{a_1}...{\bar{x_i}}^{a_i}...{x_k}^{a_k}{x_{k+1}}^{a_i}, ~ v =
{x_1}^{b_1}...{\bar{x_i}}^{b_i+1}...{x_k}^{b_k}x_{k+1}.$

This is the form (\ref{2}).

By the above analysis, when $p_1 = \pi_2(p)$ has form (\ref{n1}), 
if $p \in (W_2)_q$, then it also has to be of the form 
(\ref{n1}).

Suppose now that $p_1$ has the form (\ref{n2}). Then the local
equations of $Z$ at $p_1$ are $x_i = x_j = 0$ for some $1 \leq i,j
\leq k$. 

Then as in the above analysis 
there exist regular parameters 
$y_1,...,...,y_n$ in ${\cal O}_{X_1,p_1}$ and
unit series $\delta_i \in $${\cal \hat{O}}_{X_1,p_1}$ such that 
$y_i = \delta_i x_i$ for $1 \leq i \leq k$.

Then ${\cal {O}}_{X_2,p}$ has one of the following two forms: 

\begin{itemize}
\item[(a)] ${\cal {O}}_{X_2,p} = {{\cal {O}}_{X_1,p_1}[\frac{y_i}{y_j}]}_{(y_j,\frac{y_i}{y_j}-\alpha)}$ 
for some $\alpha \in K$, or 
\item[(b)] ${\cal {O}}_{X_2,p} = {{\cal {O}}_{X_1,p_1}[\frac{y_j}{y_i}]}_{(y_i,\frac{y_j}{y_j})}$
\end{itemize}

Arguing as above in case (a) we obtain regular parameters $x_1,...,\bar{x}_i,...,x_n$ in 
${\cal \hat{O}}_{X_2,p}$ so that 
$$u =
{x_1}^{a_1}...{{(\bar{x}}_i+\alpha)}^{a_i}...{x_j}^{a_i+a_j}...{x_k}^{a_k},
v =
{x_1}^{b_1}...{{(\bar{x}}_i+\alpha)}^{b_i}...{x_j}^{b_i+b_j}...{x_k}^{b_k}.$$
This is the form (\ref{5}) if $\alpha = 0$.

If $\alpha \neq 0$, we obtain either the form (\ref{5}) or the form
(\ref{4}) according as rank of \\ $\left[ \begin{array}{ccccccccccc} a_1
& .&.& a_i+a_j& .&.& a_{j-1}&a_{j+1}&.&.& a_k
\\b_1 &.&.& b_i+b_j&.&.& b_{j-1}&b_{j+1}&.&.& b_k
\end{array} \right]$ is $=$ 2 or $<$ 2.

Again arguing as above in case (b) we obtain regular parameters $x_1,...,\bar{x}_j,...,x_n$ in 
${\cal \hat{O}}_{X_2,p}$ so that 
$$u =
{x_1}^{a_1}...{x_i}^{a_i+a_j}...{\bar{x_j}}^{a_j}...{x_k}^{a_k},
v =
{x_1}^{b_1}...{x_i}^{b_i+b_j}...{\bar{x_j}}^{b_j}...{x_k}^{b_k}.$$
This is the form (\ref{5}).

By the above analysis, when $p_1 = \pi_2(p)$ has the form (\ref{n2}), 
if $p \in (W_2)_q$, then it also has to be of the form 
(\ref{n2}).

This completes the proof of {\bf I.A} for $X_2$. Now {\bf I.B} is clear as the forms 
(\ref{n1}) and (\ref{n2}) are just the forms 
(\ref{1}) and (\ref{5}) from {\bf I.A}. \\

\noindent{\bf II} 
We prove this part by induction on the number of blowups in the 
sequence $\pi_1: X_1 \rightarrow X$. 

Since $q \notin E_i$ and $f$ is locally toroidal,
$f$ is smooth 
at any point $p_1 \in f^{-1}(q)$. This means that the regular parameters at $q$ 
form a part of a regular sequence at $p$. So we have 
regular parameters \x~   in ${\cal \hat{O}}_{X,p_1}$ and
$u,v$ in \oy such that 
$u = x_1, v = x_2.$  This is the form (\ref{6}).
Thus the conclusions hold in $X$. 
Suppose that the conclusions of the lemma
hold after any sequence of $l$ permissible blowups where $l \geq 0$.  

Let $\pi_1: X_1 \rightarrow X$ be a permissible sequence (with respect to
$q$) of $l$ blowups. Let $\pi_2: X_2 \rightarrow X_1$ be the blowup of 
a nonsingular codimension 2 subvariety $Z$ of $X_1$ such that $Z \subset (W_1)_q$. 

Let $p \in {\pi_2}^{-1}({\pi_1}^{-1}(U_i)) \cap {(f \circ \pi_1 \circ \pi_2)}^{-1}(q)$ for 
some $1 \leq i \leq m$.

If $p_1 = \pi_2(p) \notin Z$ then $\pi_2$ is an isomorphism at $p$ and
we have nothing to
prove. Suppose then that $p_1 \in {\pi_1}^{-1}(U_i) \cap Z \subset 
{\pi_1}^{-1}(U_i) \cap (W_1)_q$.

Then by induction hypothesis ({\bf II.B}) 
$p_1$ has the form (\ref{n3}).  
Then the local equations of $Z$
at $p_1$ are $x_1 = x_2 = 0$.  

There exist regular parameters $\bar{x}_1,\bar{x}_2$ in 
${\cal \hat{O}}_{X_2,p}$ such that one of the
following forms holds:

$x_1 = \bar{x}_1, x_2 = \bar{x}_1(\bar{x}_2+\alpha)$ for some 
$\alpha \in K$ or $x_1 = \bar{x}_1\bar{x}_2, x_2 = \bar{x}_2$.
These two cases give the forms (\ref{7}) and (\ref{8}).

Now {\bf II.B} is clear as the form
(\ref{n3}) is just the form (\ref{6}) from {\bf II.A}. \\

\noindent
{\bf III} Since $\{{\pi_1}^{-1}(U_i)\}$ for $1 \leq i \leq m$
is an open cover of $X_1$ and ${\pi_1}^{-1}(U_i) \cap (W_1)_q$
is a union of nonsingular codimension 2 subvarieties of $X_1$ for all
$i$ by {\bf I} and {\bf II}, $(W_1)_q$ is a union of nonsingular codimension
2 subvarieties of $X_1$.
\end{proof}

\section{Principalization}
Let \map be a locally toroidal morphism from a nonsingular
$n$-fold $X$ to a nonsingular surface $Y$ with respect to 
open coverings $\{U_1,...,U_m\}$ and $\{V_1,...,V_m\}$ 
of $X$ and $Y$ respectively and SNC divisors 
$D_i$ and $E_i$  
in $U_i$ and $V_i$ respectively.

In this section we fix an $i$ between 
1 and $m$ and a $q \in Y$.

Let $\pi_1: X_1 \rightarrow X$ be a permissible sequence with respect to
$q$. Our aim is to construct a permissible sequence $\pi_2:X_2 \rightarrow X_1$ such that 
$\pi_2 \circ \pi_1: X_2 \rightarrow X$ is a permissible sequence and
${\pi_2}^{-1}({\pi_1}^{-1}(U_i)) \cap (W_2)_q$ is empty.

First suppose that $q \notin E_i$. If 
$p \in {\pi_1}^{-1}(U_i)$, then by Lemma \ref{a} one of
the forms (\ref{6}), (\ref{7}) or (\ref{8}) holds at $p$.  





\begin{theorem}\label{not in E_i} Let $\pi_1: X_1 \rightarrow X$ be a 
permissible sequence with respect to $q \in Y$. 
Let $i$ be such that $q \notin E_i$. 
Then there exists a permissible sequence $\pi_2: X_2 \rightarrow X_1$
with respect to $q$
such that 
${\pi_2}^{-1}({\pi_1}^{-1}(U_i)) \cap (W_2)_q$ is empty.
\end{theorem}
\begin{proof}

If ${\pi_1}^{-1}(U_i) \cap (W_2)_q$ is empty, then there is nothing
to prove. So suppose that ${\pi_1}^{-1}(U_i) \cap (W_2)_q \ne \emptyset$.
By Lemma \ref{ni}, it is a union of codimension 2 subvarieties of 
${\pi_1}^{-1}(U_i)$.

Let
$Z \subset {\pi_1}^{-1}(U_i) \cap (W_1)_q$ be a 
subvariety of ${\pi_1}^{-1}(U_i)$ of codimension 2.

Let $\pi_2: X_2 \rightarrow X_1$ be the blowup of the Zariski closure
$\bar{Z}$ of $Z$ in $X_1$. Let $Z_1 \subset {\pi_2}^{-1}(Z)$
be a codimension 2 subvariety of ${\pi_2}^{-1}({\pi_1}^{-1}(U_i))$ such that 
$Z_1 \subset {\pi_2}^{-1}({\pi_1}^{-1}(U_i)) \cap (W_2)_q$.

By the proof of Lemma \ref{a} it follows that 
$Z_1 \cap (W_2)_q = \emptyset$. 

The theorem now follows by induction on the number of codimension 2 subvarieties 
$Z$ in 
${\pi_1}^{-1}(U_i) \cap (W_1)_q $.
\end{proof}

Now we suppose that $q \in E_i$.

\begin{remark}\label{1 point goes to 1 point}
Suppose that $\pi_1: X_1 \rightarrow X$ is a permissible sequence with 
respect to some $q \in E_i$. 
Let $\pi_2: X_2 \rightarrow X_1$ be a permissible blowup with respect 
to $q$. Let
$p_1 \in {\pi_2}^{-1}({\pi_1}^{-1}(U_i)) \cap (W_2)_q$. Then 
clearly $p = \pi_2(p_1) \in {\pi_1}^{-1}(U_i) \cap (W_1)_q$.

Suppose that $p_1$ is a 1 point. Then the analysis in the proof 
of Lemma \ref{a} shows that $p$ also is a 1 point. 

Suppose that $p_1$ is a 2 point where the form (\ref{n2}) holds. 
Then the analysis in the proof of Lemma \ref{a} shows that
$p$ is a 2 or 3 point where the from (\ref{n2}) holds. 
\end{remark}

Suppose that $\pi_1: X_1 \rightarrow X$ is a permissible sequence
with respect to $q \in E_i$. 

Let $p \in {\pi_1}^{-1}(U_i) \cap (W_1)_q$ be a 1 point. By Lemma \ref{a}, there exist regular parameters
$x_1,...,x_n$ in ${\cal \hat{O}}_{X_1,p}$ and $u,v$ in \oy such that $u = {x_1}^a,~ v = {x_1}^bx_2$ where
$a > b$.

Define $\Omega_i(p) = a-b > 0.$

Let $Z \subset {\pi_1}^{-1}(U_i) \cap (W_1)_q$ be a codimension 2 subvariety of ${\pi_1}^{-1}(U_i)$.

Define $\Omega_i(Z) = \Omega_i(p)$ if there exists a 1 point $p \in Z$. This is well defined because
$\Omega_i(p) = \Omega_i(p^{\prime})$ for any two points $p, p^{\prime} \in Z$.

If $Z$ contains no 1 points, we define $\Omega_i(Z) = 0$.

Finally define 
\begin{eqnarray*}
\Omega_i(f \circ \pi_1) = max\{\Omega_i(Z) | Z \subset {\pi_1}^{-1}(U_i) \cap 
 (W_1)_q {\rm ~is~ an~ irreducible}\\
{\rm ~subvariety ~of~} {\pi_1}^{-1}(U_i) {\rm ~of~ codimension~} 2\}
\end{eqnarray*}

\begin{theorem}\label{1 point} Let $\pi_1: X_1 \rightarrow X$ be a 
permissible sequence with respect to $q \in E_i$.
There exists a permissible sequence $\pi_2: X_2 \rightarrow X_1$
with respect to $q$ such that $\Omega_i(f \circ \pi_1 \circ \pi_2) = 0$.
\end{theorem}
\begin{proof}
Suppose that $\Omega_i(f \circ \pi_1) > 0$. Let $Z \subset {\pi_1}^{-1}(U_i) \cap (W_1)_q$ be a 
subvariety of ${\pi_1}^{-1}(U_i)$ of codimension 2 such that
$\Omega_i(f \circ \pi_1) = \Omega_i(Z)$.

Let $\pi_2: X_2 \rightarrow X_1$ be the blowup of the Zariski closure
$\bar{Z}$ of $Z$ in $X_1$. Let $Z_1 \subset {\pi_2}^{-1}(Z)$
be a codimension 2 subvariety of ${\pi_2}^{-1}({\pi_1}^{-1}(U_i))$ such that 
$Z_1 \subset {\pi_2}^{-1}({\pi_1}^{-1}(U_i)) \cap (W_2)_q$.
We claim that $\Omega_i(Z_1) < \Omega_i(Z)$.

If there are no 1 points of $Z_1$ then we have nothing to prove. Otherwise, let $p_1 \in Z_1$ be a
1 point. Then $\pi_1(p_1) = p$ is a 1 point of $Z$ by Remark \ref{1 point goes to 1 point}.

There are regular parameters $x_1,...,x_n$ in ${\cal \hat{O}}_{X_1,p}$
 and $u,v$ in \oy such that
$u = {x_1}^a, v = {x_1}^bx_2$. There exist regular parameters $x_1,\bar{x_2},...,x_n$ in
${\cal \hat{O}}_{X_2,p_1}$ such that $x_2 = x_1(x_2+\alpha)$.

$u = {x_1}^a, v = {x_1}^{b+1}(x_2+\alpha)$. Since $p_1 \in (W_2)_q$, $\alpha = 0$.

$\Omega_i(Z_1) = \Omega_i(p_1) = a-b-1 < a-b = \Omega_i(Z)$.

The theorem now follows by induction on the number of codimension 2 subvarieties $Z$ in 
${\pi_1}^{-1}(U_i) \cap (W_1)_q $
such that $\Omega_i(f \circ \pi_1) = \Omega_i(Z)$ and induction on $\Omega_i(f \circ \pi_1)$.
\end{proof}


Let $\pi_1: X_1 \rightarrow X$ be a permissible sequence with respect to $q \in E_i$. 

Let $Z \subset {\pi_1}^{-1}(U_i)) \cap (W_1)_q $
 be a codimension 2 subvariety of ${\pi_1}^{-1}(U_i)$. Let 
$p \in Z$ be a 2 point where the form (\ref{n2}) holds.

There exist regular parameters $x_1,...,x_n$ in ${\cal \hat{O}}_{X_1,p}$ and 
$u,v$ in \oy such that $u={x_1}^{a_1}{x_2}^{a_2}$ and $v = {x_1}^{b_1}{x_2}^{b_2}$.

Define $\omega_i(p) = (a_1-b_1)(b_2-a_2)$. Then since $p \in (W_1)q$, $\omega_i(p) > 0$.  

Now define $\omega_i(Z) = \omega_i(p)$ if $p \in Z$ is a 2 point where the form (\ref{n2})
holds. If there are no 2 points of the form (\ref{n2}) in 
$Z$ define $\omega_i(Z) = 0$. Then $\omega_i(Z)$ is well-defined. 

Finally define 
\begin{eqnarray*}
\omega_i(f \circ \pi_1) = 
max\{\omega_i(Z) | Z \subset {\pi_1}^{-1}(U_i) \cap (W_1)_q{\rm ~is~ an~ irreducible}\\
{\rm ~subvariety ~of~} 
{\pi_1}^{-1}(U_i) {\rm ~of~ codimension~} 2\}
\end{eqnarray*}

\begin{theorem} \label{2 point} Let $\pi_1: X_1 \rightarrow X$ be a permissible 
sequence with respect to $q \in E_i$. Suppose that $\Omega_i(f \circ \pi_1) = 0$. 
There exists a permissible sequence $\pi_2: X_2 \rightarrow X_1$  with respect to $q$
such that $\Omega_i(f \circ \pi_1 \circ \pi_2) = 0$ 
and $\omega_i(f \circ \pi_1 \circ \pi_2) = 0$.
\end{theorem}
\begin{proof}
Since $\Omega_i(f \circ \pi_1) = 0$, there are no 1 points in ${\pi_1}^{-1}(U_i) \cap (W_1)_q$.
Let $X_2 \rightarrow X_1$ be any permissible blowup. Then by 
Remark \ref{1 point goes to 1 point} it follows that 
${\pi_2}^{-1}({\pi_1}^{-1}(U_i)) \cap (W_2)_q$ has no 1 points. Hence 
$\Omega_i(f \circ \pi_1 \circ \pi_2) = 0$. 

Suppose that $\omega_i(f \circ \pi_1) > 0$. Let $Z \subset {\pi_1}^{-1}(U_i) \cap (W_1)_q$
be a codimension 2 irreducible subvariety of ${\pi_1}^{-1}(U_i)$ such that 
$\omega_i(f \circ \pi_1) = \omega_i(Z)$.

Let $\pi_2: X_2 \rightarrow X_1$ be the blowup of the Zariski closure $\bar{Z}$ of $Z$
in $X_1$. Let $Z_1 \subset {\pi_2}^{-1}(Z)$
be a codimension 2 subvariety of ${\pi_2}^{-1}({\pi_1}^{-1}(U_i))$ such that 
$Z_1 \subset {\pi_2}^{-1}({\pi_1}^{-1}(U_i)) \cap (W_2)_q$.
We prove that $\omega_i(Z_1) < \omega_i(Z) = \omega_i(f \circ \pi_1)$.

If there are no 2 points of the form (\ref{n2}) in $Z_1$ then 
$\omega_i(Z_1) = 0$ and we have nothing to prove. Otherwise let
$p_1 \in Z_1$ be a 2 point of the form (\ref{n2}). 

By Remark \ref{1 point goes to 1 point}, $p = \pi_2(p_1) \in Z$ is a 
2 or 3 point of form (\ref{n2}).

Suppose that $p \in Z$ is a 2 point. 
There exist regular parameters $x_1,...,x_n$ in ${\cal \hat{O}}_{X_1,p}$ and 
$u,v$ in \oy such that $u={x_1}^{a_1}{x_2}^{a_2}$ and $v = {x_1}^{b_1}{x_2}^{b_2}$.
Also the local equations of $Z$ are $x_1 = x_2 = 0$. 

Then there exist regular parameters $x_1,\bar{x_2},x_3...,x_n$ 
in ${\cal \hat{O}}_{X_2,p_1}$ 
such that 
$x_2 = x_1\bar{x_2}$ and 
$u={x_1}^{a_1+a_2}{\bar{x_2}}^{a_2}$ and $v = {x_1}^{b_1+b_2}{\bar{x_2}}^{b_2}$.
\begin{eqnarray*}
\omega_i(Z_1) = \omega_i(p_1) &=& (a_1+a_2-b_1-b_2)(b_2-a_2) \\
&=& (a_1-b_1)(b_2-a_2)+(a_2-b_2)(b_2-a_2)\\
&<& (a_1-b_1)(b_2-a_2) = \omega_i(p) = \omega_i(Z) = \omega_i(f \circ \pi_1).
\end{eqnarray*}
Suppose that $p \in Z$ is a 3 point. 
There exist regular parameters $x_1,...,x_n$ in ${\cal \hat{O}}_{X_1,p}$ and 
$u,v$ in \oy such that $u={x_1}^{a_1}{x_2}^{a_2}{x_3}^{a_3}$ and 
$v = {x_1}^{b_1}{x_2}^{b_2}{x_3}^{b_3}$. After permuting $x_1,x_2,x_3$ if necessary,
we can 
suppose that the local equations of $Z$ are $x_2 = x_3 = 0$.

Then there exist regular parameters $x_1,x_2,\bar{x_3}...,x_n$ 
in ${\cal \hat{O}}_{X_2,p_1}$ 
such that 
$x_3 = x_2(\bar{x_3}+\alpha)$ and 
$u={x_1}^{a_1}{x_2}^{a_2+a_3}({\bar{x_3}+\alpha})^{a_3}$ and 
$v = {x_1}^{b_1}{x_2}^{b_2+b_3}({\bar{x_3}+\alpha})^{b_3}$.

Since $p_1$ is a 2 point, we have $\alpha \neq 0$ and 
$a_1(b_2+b_3) - b_1(a_2+a_3) \neq 0$. After an appropriate change of 
variables $x_1,x_2$ we obtain regular parameters 
$\bar{x_1},\bar{x_2},\tilde{x_3},x_4,...,x_n$ in 
${\cal \hat{O}}_{X_2,p_1}$. 

$u={\bar{x_1}}^{a_1}{\bar{x_2}}^{a_2+a_3}$ and 
$v = {\bar{x_1}}^{b_1}{\bar{x_2}}^{b_2+b_3}$.

Since the local equations of $Z \subset {\pi_1}^{-1}(U_i) \cap (W_1)_q$ 
are $x_2 = x_3 = 0$, $b_2-a_2$ and $b_3-a_3$ have different signs. 
So $a_1-b_1$ has the same sign as exactly one of $b_2-a_2$ or $b_3-a_3$.
Without loss of generality suppose that $(a_1-b_1)(b_2-a_2) > 0$
and $(a_1-b_1)(b_3-a_3) < 0$. 

Let $Z^{\prime}$ be the codimension 2 variety whose local
equations are $x_1=x_2=0$ defined in an appropriately small neighborhood 
in ${\pi_1}^{-1}(U_i)$.
Then the closure $\bar{Z^{\prime}}$ of $Z^{\prime}$ in ${\pi_1}^{-1}(U_i)$ is 
an irreducible codimension 2 subvariety 
contained 
in ${\pi_1}^{-1}(U_i) \cap (W_1)_q$.
\begin{eqnarray*}
\omega_i(Z_1) = \omega_i(p_1) &=& (a_1-b_1)(b_2+b_3-a_2-a_3) \\
&=& (a_1-b_1)(b_2-a_2)+(a_1-b_1)(b_3-a_3)\\
&<& (a_1-b_1)(b_2-a_2) = \omega_i(\bar{Z^{\prime}}) \leq \omega_i(f \circ \pi_1).
\end{eqnarray*}
The theorem now follows by induction on the number of codimension 2 subvarieties $Z$ in 
${\pi_1}^{-1}(U_i) \cap (W_1)_q $
such that $\omega_i(f \circ \pi_1) = \omega_i(Z)$ and induction on $\omega_i(f \circ \pi_1)$.
\end{proof}

\begin{remark} \label{no 2 points} Let $\pi_1: X_1 \rightarrow X$ be a permissible
sequence with respect to $q$. Let $i$ be 
such that 
$1 \leq i \leq m$.

If $q \in E_i$, then
it follows from Theorems \ref{1 point} and 
\ref{2 point} that there exists a permissible sequence 
$\pi_2: X_2 \rightarrow X_1$ 
with respect to $q$
such that $\Omega_i(f \circ \pi_1 \circ \pi_2) = 0$ and 
$\omega_i(f \circ \pi_1 \circ \pi_2) = 0$.
\end{remark}

\begin{theorem}\label{principalization} Let \map be a locally toroidal morphism 
between
a nonsingular $n$-fold $X$ and a nonsingular surface $Y$. Let $q \in Y$.

Then there exists a permissible sequence $\pi_1: X_1
\rightarrow X$ with respect to $q$ such that $(W_1)_q$ is empty.
\end{theorem}
\begin{proof} First we apply 
Theorem \ref{not in E_i} and Remark \ref{no 2 points}
to $X$ and $i = 1$.

Suppose that $q \notin E_1$. Then 
by Theorem \ref{not in E_i},
there exists 
a permissible sequence $\pi_1:X_1 \rightarrow X$ with 
respect to $q$ such that 
${\pi_1}^{-1}(U_1) \cap (W_1)_q = \emptyset$.

Now suppose that $q \in E_1$. 
It follows from Remark \ref{no 2 points} 
that there exists 
a permissible sequence $\pi_1:X_1 \rightarrow X$ with respect to $q$
such that 
$\Omega_1(f \circ \pi_1) = 0$ and 
$\omega_1(f \circ \pi_1) = 0$.
So there are no 1 points or 2 points of the
form (\ref{n2})
in
${\pi_1}^{-1}(U_1) \cap (W_1)_q$. But if $Z \subset {\pi_1}^{-1}(U_1) \cap (W_1)_q$ is any codimension 2 
irreducible subvariety of ${\pi_1}^{-1}(U_i)$, then a generic point of $Z$ must either be a 1 point 
or a 2 point of the form (\ref{n2}). 	
It follows then that ${\pi_1}^{-1}(U_1) \cap (W_1)_q$ is empty. 

Now we apply Theorem \ref{not in E_i} and
Remark \ref{no 2 points} to the permissible sequence $\pi_1: X_1 \rightarrow X$
and $i = 2$. 

If $q \notin E_2$, then by 
Theorem \ref{not in E_i}, there exists a permissible sequence 
$\pi_2: X_2 \rightarrow X_1$ such that 
${\pi_2}^{-1}({\pi_1}^{-1}(U_2)) \cap (W_2)_q = \emptyset$.

If $q \in E_2$, then as above there exists a permissible sequence 
$\pi_2: X_2 \rightarrow X_1$ such that 
${\pi_2}^{-1}({\pi_1}^{-1}(U_2)) \cap (W_2)_q$ is empty. 

Notice that we also have 
${\pi_2}^{-1}({\pi_1}^{-1}(U_1)) \cap (W_2)_q = \emptyset$.

Repeating the argument for $i = 3,4,...,m$ we obtain the desired permissible sequence.
\end{proof}

\section{Toroidalization}
\begin{theorem}	\label{final} 
Let \map be a locally toroidal morphism from a nonsingular
$n$-fold $X$ to a nonsingular surface $Y$ with respect to 
open coverings $\{U_1,...,U_m\}$ and $\{V_1,...,V_m\}$ 
of $X$ and $Y$ respectively and SNC divisors 
$D_i$ and $E_i$  
in $U_i$ and $V_i$ respectively. 
Let $\pi:Y_1 \rightarrow Y$ be the blowup
of a point $q \in Y$.

Then there exists a permissible sequence
$\pi_1:X_1 \rightarrow X$ such that there is a 
locally toroidal morphism $f_1:X_1
\rightarrow Y_1$ such that $\pi \circ f_1 = f \circ \pi_1$.


\end{theorem}
\begin{proof} By Theorem \ref{principalization} there is a
permissible sequence $\pi_1: X_1 \rightarrow X$ such that there exists a morphism
$f_1: X_1 \rightarrow Y_1$ and $\pi \circ f_1 = f \circ \pi_1$.

Let $p \in X_1$. Suppose that 
$p \in {\pi_1}^{-1}(U_i)$ for 
some $i$ such that $1 \leq i \leq m$. 
If $\pi_1(p) \notin f^{-1}(q)$ then
we have nothing to prove. So we assume that $\pi_1(p) \in
{f}^{-1}(q)$.

Suppose first that $q \notin E_i$. Then 
by Lemma \ref{a} one of the forms 
(\ref{7}) or (\ref{8}) holds at $p$.
So there exist regular 
parameters $x_1,...,x_n$ in ${\cal \hat{O}}_{X_1,p}$
and $u,v$ in ${\cal {O}}_{Y,q}$ such that $$u = {x_1},
~v = {x_1}(x_2+\alpha)~{\rm  for~ some~} \alpha \in K,{\rm ~ or~}
u = x_1y_1, v = x_2.$$

Let $f_1(p) = q_1$. There exist regular parameters
$u_1,v_1 \in {\cal O}_{Y_1,q_1}$ such that 
$$u = u_1, v = u_1(v_1+\alpha) ~{\rm or}~ u = u_1v_1, v = v_1$$
according as the form (\ref{7}) 
or the form (\ref{8}) holds.  
In either case, we have
$u_1 = x_1, v_1 = x_2$,
and $f_1$ is smooth at $p$.

Now suppose that $q \in E_i$.

By Lemma \ref{a} there exist regular parameters \x ~ in ${\cal
\hat{O}}_{X_1,p}$ and $u,v$ in \oy such that one of the forms 
(\ref{1}), (\ref{2}), (\ref{3}), (\ref{4}), or (\ref{5})
of Lemma \ref{a} holds.

Suppose first that the form (\ref{1}) holds. Then since $m_q{\cal \hat{O}}_{X_1,p}$
is invertible, there exist regular parameters $x_1,...,x_n$ in ${\cal \hat{O}}_{X_1,p}$
and $u,v$ in ${\cal {O}}_{Y,q}$ such that $u = {x_1}^{a_1}...{x_k}^{a_k},
~v = {x_1}^{a_1}...{x_k}^{a_k}x_{k+1}$ for some $1 \leq k \leq n-1$.

Further $x_1...x_k = 0$ is a local equation of
${\pi_1}^{-1}(D_i)$ and $u = 0$ is a local equation for $E_i$. 

Let $f_1(p) = q_1$. There exist regular parameters $(u_1,v_1)$ in
${\cal O}_{Y_1,q_1}$ such that $u=u_1$ and $v=u_1v_1$. Hence the
local equation of ${\pi}^{-1}(E_i)$ at $q_1$ is $u_1 = 0$. $$u_1
= {x_1}^{a_1}...{x_k}^{a_k}, v_1= x_{k+1}.$$ This is the form
(\ref{t1}). 

Suppose now that the form (\ref{2}) holds at $p$ for $f \circ
\pi_1$. There exist regular 
parameters $x_1,...,x_n$ in ${\cal \hat{O}}_{X_1,p}$
and $u,v$ in ${\cal {O}}_{Y,q}$  and $1 \leq k \leq n-1$
such that 
$u = 0$ is a local equation of $E_i$,
$x_1...x_kx_{k+1} = 0$ is a local equation of
${\pi_1}^{-1}(D_i)$ and $$u =
{x_1}^{a_1}...{x_k}^{a_k}{x_{k+1}}^{a_{k+1}}, v =
{x_1}^{b_1}...{x_k}^{b_k}{x_{k+1}}^{b_{k+1}},$$ where $b_i \leq a_i$
for $i = 1,...,k$ and $b_{k+1} < a_{k+1}$.

Let $f_1(p) = q_1$. There exist regular parameters $u_1,v_1$ in
${\cal O}_{Y_1,q_1}$ 
such that $u=u_1v_1$ and $v=v_1$. Hence the
local equation of ${\pi}^{-1}(E_i)$ at $q_1$ is $u_1v_1 = 0$. $$u_1
= {x_1}^{a_1-b_1}...{x_k}^{a_k-b_k}{x_{k+1}}^{a_{k+1}-b_{k+1}}, v_1
= {x_1}^{b_1}...{x_k}^{b_k}{x_{k+1}}^{b_{k+1}}.$$ This is the form
(\ref{t3}). Note that the rank condition follows from the dominance of
the map $f_1$.

Suppose now that the form (\ref{3}) holds. There exist regular parameters 
$x_1,...,x_n$ in ${\cal \hat{O}}_{X_1,p}$
and $u,v$ in ${\cal {O}}_{Y,q}$ and $1 \leq k \leq n-1$
such that 
 $u = 0$ is a local
equation of $E_i$, $x_1...x_k = 0$ is a local
equation of ${\pi_1}^{-1}(D_i)$ and $$u = {x_1}^{a_1}...{x_k}^{a_k},
v = {x_1}^{b_1}...{x_k}^{b_k}(x_{k+1}+\alpha),$$ where $b_i \leq
a_i$ for all $i$ and $0 \neq \alpha \in K$.

Let $f_1(p) = q_1$. There exist regular parameters $u_1,v_1$ in
${\cal O}_{Y_1,q_1}$ such that $u=u_1v_1$ and $v=v_1$. Hence the
local equation of ${\pi}^{-1}(E_i)$ at $q_1$ is $u_1v_1 = 0$.
$$u_1 = {x_1}^{a_1-b_1}...{x_k}^{a_k-b_k}{(x_{k+1}+\alpha)}^{-1},
v_1 = {x_1}^{b_1}...{x_k}^{b_k}(x_{k+1}+\alpha).$$

If rank $\left[ \begin{array}{cccc} a_1-b_1 & .&.& a_k-b_k \\b_1
&.&.& b_k
\end{array} \right]$= 2 then there exist regular parameters
\xx~ in ${\cal \hat{O}}_{X_1,p}$ such that $u_1 =
{\bar{x_1}}^{a_1-b_1}...{\bar{x_k}}^{a_k-b_k}, v_1 =
{\bar{x_1}}^{b_1}...{\bar{x_k}}^{b_k}.$ This is the form 
(\ref{t3}).

If rank $\left[ \begin{array}{cccc} a_1-b_1 & .&.& a_k-b_k \\b_1
&.&.& b_k
\end{array} \right]$ $<$  2 then there exist regular parameters
\xx~ in ${\cal \hat{O}}_{X_1,p}$ such that $u_1 =
{({\bar{x_1}}^{a_1}...{\bar{x_k}}^{a_k})}^m, ~ v =
{({\bar{x_1}}^{a_1}...{\bar{x_k}}^{a_k})}^t(x_{k+1}+\beta)$, with
$\beta \neq 0$. This is the form (\ref{t2}).

Suppose that the form (\ref{4}) holds. There exist regular parameters 
$x_1,...,x_n$ in ${\cal \hat{O}}_{X_1,p}$
and $u,v$ in ${\cal {O}}_{Y,q}$ 
and $1 \leq k \leq n-1$
such that 
 $uv = 0$ is a
local equation for $E_i$, $x_1...x_k = 0$ is
a local equation of ${\pi_1}^{-1}(D_i)$ and $$u =
{({x_1}^{a_1}...{x_k}^{a_k})}^m, ~ v =
{({x_1}^{a_1}...{x_k}^{a_k})}^t(\alpha+x_{k+1}),$$ where
$a_1,...,a_k, m, t> 0$ and $\alpha \in K - \{0\}$.

Suppose that $m \leq t$. There exist regular parameters $u_1,v_1$
in ${\cal O}_{Y_1,q_1}$ such that $u=u_1$ and $v=u_1(v_1+\beta)$ for
some $\beta \in K$. $$u_1 = {({x_1}^{a_1}...{x_k}^{a_k})}^m, ~ v_1 =
{({x_1}^{a_1}...{x_k}^{a_k})}^{t-m}(\alpha+x_{k+1})-\beta.$$

If $m<t$ then $\beta = 0$. So $u_1v_1=0$ is a local equation of
${\pi}^{-1}(E_i)$ and we have the form (\ref{t2}). If $m=t$ then
$\alpha = \beta \neq 0$ and $u_1$ is a local equation of
${\pi}^{-1}(E_i)$. In this case we have the form (\ref{t1}).

Suppose that $m>t$. Then there exist regular parameters $u_1,v_1$
in ${\cal O}_{Y_1,q_1}$ such that $u=u_1v_1$ and $v=v_1$. $$u_1 =
{({x_1}^{a_1}...{x_k}^{a_k})}^{m-t}{(\alpha+x_{k+1})}^{-1}, ~ v_1 =
{({x_1}^{a_1}...{x_k}^{a_k})}^t(\alpha+x_{k+1}).$$ We obtain the
form (\ref{t2}).

Finally suppose that the form (\ref{5}) holds. There exist regular parameters 
$x_1,...,x_n$ in ${\cal \hat{O}}_{X_1,p}$ and $u,v$ in ${\cal {O}}_{Y,q}$ 
and $2 \leq k \leq n$
such that $uv = 0$ is a local equation of $E_i$ and 
$x_1...x_k = 0$ is a local equation of ${\pi_1}^{-1}(D_i)$ and 
$u = {x_1}^{a_1}...{x_k}^{a_k},
v = {x_1}^{b_1}...{x_k}^{b_k},$ where 
rank $\left[
\begin{array}{cccc} a_1 & .&.& a_k \\b_1 &.&.& b_k \end{array}
\right]$ = 2 .

We have either $a_i \geq b_i$ for all $i$ or $a_i \leq b_i$ for all $i$. 
Without loss of generality, suppose that $a_i \leq b_i$ for all $i$. 

Let $f_1(p) = q_1$. There exist regular parameters $u_1,v_1$ in
${\cal O}_{Y_1,q_1}$ 
such that $u=u_1$ and $v=u_1v_1$. Hence the
local equation of ${\pi}^{-1}(E_i)$ at $q_1$ is $u_1v_1 = 0$. $$u_1
= {x_1}^{a_1}...{x_k}^{a_k}, v_1
= {x_1}^{b_1-a_1}...{x_k}^{b_k-a_k}.$$ 
Further, rank $\left[
\begin{array}{cccc} a_1 & .&.& a_k \\b_1-a_1 &.&.& b_k-a_k \end{array}
\right]$ = 2. This is the form
(\ref{t1}). 
\end{proof}

Now we are ready to prove our main theorem.

\begin{theorem}\label{Final}
Suppose that \map is a locally toroidal morphism between
a variety $X$ and a surface $Y$. Then there exists a 
commutative diagram of morphisms 
\begin{displaymath}
\xymatrix @R=3pc @C=3pc{
X_1 \ar[r]^{f_1} \ar[d]^{\pi_1} & Y_1 \ar[d]^{\pi} \\
X  \ar[r]^f    & Y }
\end{displaymath}
where $\pi$, $\pi_1$ are blowups of nonsingular varieties such that 
there exist 
SNC 
divisors $E$, $D$ on $Y_1$ and $X_1$ respectively such that 
$Sing(f_1) \subset D$, ${f_1}^{-1}(E) = D$ and $f_1$
is toroidal with respect to $E$ and $D$.


\end{theorem}
\begin{proof} Let $E^{\prime} = \bar{E_1} + ... + \bar{E_m}$ where
$\bar{E_i}$ is the Zariski closure of $E_i$ in
$Y$. 
There exists a finite sequence of
blowups of points $\pi: Y_1 \rightarrow Y$ such that 
${\pi}^{-1}(E^{\prime})$ 
is a 
SNC divisor on $Y_1$. 

By Theorem \ref{final}, there exists a 
sequence of blowups $\pi_1:X_1 \rightarrow X$ such that there
is a locally toroidal morphism $f_1: X_1 \rightarrow Y_1$ with 
$f \circ \pi_1 = \pi \circ f_1$.
 
Let $E = {\pi}^{-1}(E^{\prime})$ and $D = {f_1}^{-1}(E)$.

We now verify that $E$ and $D$
are SNC divisors on $Y_1$ and $X_1$ respectively and that
$f_1: X_1 \rightarrow Y_1$ is toroidal with respect to $D$ and $E$.

Let $p \in X_1$ and let $q = f_1(p)$. 

Suppose that $p \notin D$, so that $q \notin E$. 
There exists $i$ such that $1 \leq i \leq m$ and $p \in {\pi_1}^{-1}(U_i)$. Then 
$q \notin E = {\pi}^{-1}(E^{\prime})$ $\Rightarrow$ 
$q \notin {\pi}^{-1}(E_i)$. 
So  
$p \notin {f_1}^{-1}({\pi}^{-1}(E_i)) = {{\pi}_1}^{-1}(D_i)$. 
Then $f_1$ is smooth at $p$ because 
$f_1|_{{\pi_1}^{-1}(U_i)}$ is toroidal.

Thus $Sing(f_1) \subset D$.

Suppose now that $p \in D$. Let $p \in {\pi_1}^{-1}(U_i)$ for some $i$ 
between 1 and $m$. 
If $q \notin {\pi}^{-1}(E_i)$ then $f_1$ is smooth at $p$
and then 
$D = {f_1}^{-1}(E)$ is a SNC divisor at $p$.  We assume then that $q \in {\pi}^{-1}(E_i)$. \\

\noindent
{\bf Case 1} $q \in E$ is a 1 point. 

$q$ is necessarily a 1 point of ${\pi}^{-1}(E_i)$. 

Then ${\pi}^{-1}(E_i)$ and $E$ are equal in a neighborhood of $q$. Hence 
${\pi_1}^{-1}(D_i)$ and $D$ are equal in a neighborhood of $p$. Since 
${\pi_1}^{-1}(D_i)$ is a SNC divisor in a neighborhood of $p$, $D$ is a 
SNC divisor in a neighborhood of $p$.

Since $f_1|_{{\pi_1}^{-1}(U_i)}$ is toroidal
there exist regular parameters $u,v$ in  ${\cal O}_{Y_1,q}$ 
and regular parameters $x_1,...,x_n$ in 
${\cal \hat{O}}_{X_1,p}$
such that the 
the form (\ref{t1}) holds at $p$ with respect to $E$ and $D$. 






\noindent
{\bf Case 2} $q \in E$ is a 2 point. 

$q$ is either a 1 point or a 2 point of ${\pi}^{-1}(E_i)$. \\

{\bf Case 2(a)} $q$ is a 1 point of ${\pi}^{-1}(E_i)$.

There exists regular parameters 
$u,v$ in ${\cal O}_{Y_1,q}$ 
and regular parameters $x_1,...,x_n$ in 
${\cal \hat{O}}_{X_1,p}$
such that the form (\ref{t1})
holds at $p$.        
There exists $\tilde{v} \in {\cal O}_{Y_1,q}$ such that 
$u,\tilde{v}$ are regular parameters 
in ${\cal O}_{Y_1,q}$,
${u}\tilde{v} = 0$ is a local equation for $E$ at $q$,
$u = 0$ is a local equation of ${\pi}^{-1}(E_i)$ at $q$,
and  
$$\tilde{v} = \alpha u + \beta v + {\rm ~higher~ degree ~terms~ in~
} u {\rm ~and~} v,$$
for some $\beta \in K$ with $\beta \neq 0$.

Since ${\pi_1}^{-1}(D_i)$ is a SNC divisor in a neighborhood of $p$, 
there exist regular parameters $\bar{x}_1,...,\bar{x}_n$ in
${\cal O}_{X_1,p}$ such that $\bar{x}_1...\bar{x}_k = 0$ is
a local equation of 
${\pi_1}^{-1}(D_i)$ at $p$. Since 
$x_1...x_k = 0$ is also a local equation of 
${\pi_1}^{-1}(D_i)$ at $p$, there exist units
$\delta_1,...,\delta_k \in {\cal \hat{O}}_{X_1,p}$
such that, 
after possibly permuting the $x_j$,
$x_j = \delta_j \bar{x}_j$ for $1 \leq j \leq k$.

\begin{eqnarray*}
\tilde{v} &=& \alpha u + \beta v + {\rm ~higher~ degree ~terms~ in~
} u {\rm ~and~} v \\
& = & \alpha {x_1}^{a_1}...{x_k}^{a_k} + \beta x_{k+1} + 
{\rm ~higher~ degree ~terms~ in~
} u {\rm ~and~} v \\
& = & \alpha {\delta_1}^{a_1}...{\delta_k}^{a_k}
{\bar{x}_1}^{a_1}...{\bar{x}_k}^{a_k} + \beta x_{k+1} +
{\rm ~higher~ degree ~terms~ in~
} u {\rm ~and~} v \\
\end{eqnarray*}

Let $\mathfrak{m}$ be the maximal ideal of 
${\cal O}_{X_1,p}$ and let $\hat{{\mathfrak{m}}} = \mathfrak{m}{\cal \hat{O}}_{X_1,p}$
be the maximal ideal of ${\cal \hat{O}}_{X_1,p}$.

Since $\beta \neq 0$, $\bar{x_1},...,\bar{x_k},\tilde{v}$ are linearly independent 
in $\hat{\mathfrak{m}}/{\hat{\mathfrak{m}}}^2 \cong \mathfrak{m}/{\mathfrak{m}}^2$,
so that they extend to a system of regular parameters in 
${\cal O}_{X_1,p}$. 

Say 
$\bar{x_1},...,\bar{x_k},\tilde{v},\tilde{x}_{k+2},...,\tilde{x}_n$.

$u\tilde{v} = \bar{x_1}...\bar{x_k}\tilde{v} = 0$ is a local equation of $D$
at $p$, so $D$ is a SNC divisor in a neighborhood of $p$, and $u, \tilde{v}$
give the form (\ref{t3}) with respect to the formal parameters 
$x_1,...,x_k,\tilde{v},\tilde{x}_{k+2},...,\tilde{x}_n$.\\






{\bf Case 2(b)} $q$ is a 2 point of ${\pi}^{-1}(E_i)$.

Then ${\pi}^{-1}(E_i)$ and $E$ are equal in a neighborhood of $q$. Hence 
${\pi_1}^{-1}(D_i)$ and $D$ are equal in a neighborhood of $p$. Since 
${\pi_1}^{-1}(D_i)$ is a SNC divisor in a neighborhood of $p$, $D$ is a 
SNC divisor in a neighborhood of $p$.

Since $f_1|_{{\pi_1}^{-1}(U_i)}$ is toroidal
there exist regular parameters 
$u,v$ in  ${\cal O}_{Y_1,q}$ and regular parameters $x_1,...,x_n$ in ${\cal \hat{O}}_{X_1,p}$
such that the 
one of the forms (\ref{t2}) or (\ref{t3})
holds at $p$ with respect to $E$ and $D$.

\end{proof}

\noindent
{\bf Acknowledgment:} I am sincerely grateful to my advisor Dale
Cutkosky for his continued support and help with this work.

\bibliographystyle{plain}

\end{document}